\documentclass[reqno, 10pt]{amsart}
\usepackage{amssymb}
\usepackage{amsmath}
\usepackage[mathscr]{euscript}
\usepackage{hyperref}
\usepackage{graphicx}
\usepackage[normalem]{ulem}

\usepackage{amssymb}
\usepackage{hyperref}
\RequirePackage[dvipsnames]{xcolor} 
\definecolor{halfgray}{gray}{0.55}
\definecolor{webgreen}{rgb}{0,0.5,0}
\definecolor{webbrown}{rgb}{.6,0,0} \hypersetup{%
  colorlinks=true, linktocpage=true, pdfstartpage=3,
  pdfstartview=FitV,%
  breaklinks=true, pdfpagemode=UseNone, pageanchor=true,
  pdfpagemode=UseOutlines,%
  plainpages=false, bookmarksnumbered, bookmarksopen=true,
  bookmarksopenlevel=1,%
  hypertexnames=true,
  pdfhighlight=/O,
  urlcolor=webbrown, linkcolor=RoyalBlue,
  citecolor=webgreen, 
  pdftitle={},%
  pdfauthor={},%
  pdfsubject={2000 MAthematical Subject Classification: Primary:},%
  pdfkeywords={},%
  pdfcreator={pdfLaTeX},%
  pdfproducer={LaTeX with hyperref}%
}

\usepackage{enumitem}

\newtheorem{theorem}{Theorem}[section]
\newtheorem{lemma}[theorem]{Lemma}
\newtheorem{corollary}[theorem]{Corollary}
\newtheorem{proposition}[theorem]{Proposition}

\newtheorem{remark}[theorem]{Remark}

\theoremstyle{definition}

\begin{document}

\title
[A variational principle for impulsive systems]
{Topological pressure for discontinuous semiflows and a variational principle for impulsive dynamical systems}

\author{Lucas Backes}
\author{Fagner B. Rodrigues}

\address{\noindent Departamento de Matem\'atica, Universidade Federal do Rio Grande do Sul, Av. Bento Gon\c{c}alves 9500, CEP 91509-900, Porto Alegre, RS, Brazil.}

\email{lucas.backes@ufrgs.br}
\email{fagnerbernardini@gmail.com}

\date{\today}


\begin{abstract}
We introduce four, a priori different, notions of topological pressure for possibly discontinuous semiflows acting on compact metric spaces and observe that they all agree with the classical one when restricted to the continuous setting. Moreover, for a class of \emph{impulsive semiflows}, which are examples of discontinuous systems, we prove a variational principle. As a consequence, we conclude that for this class of systems the four notions coincide and, moreover, they also coincide with the notion of topological pressure introduced in \cite{ACS17}. 
\end{abstract}

\maketitle

\section{Introduction}

The general aim of Ergodic Theory is to understand the stochastic behavior of deterministic systems and this is done by studying invariant measures. Since, in general, a dynamical system may have plenty of invariant measures, a fundamental question that arises is the following: which invariant measure should we choose to analyze the system?

Inspired by statistical mechanics, the theory of \emph{thermodynamical formalism} was introduced to the realm of Dynamical Systems by the pioneering works of Ruelle, Sinai and Bowen in \cite{BR75,Ruelle,RuelleS} and, since them, have been extensively studied by many authors (see, for instance, \cite{Baladi,AC,DG1999,Georgii,Lopes,VV,Walters1975}).  One of the main notions of this theory is that of \emph{topological pressure}. This quantity, defined in terms of topological properties, has a deep connection with the metric properties of the system. For instance, given a continuous map $T:X\to X$ acting on compact metric space and a continuous potential $f:X\to \mathbb R$, the topological pressure $P(T,f)$ of $f$ with respect to $T$ satisfies the \emph{variational principle}
\begin{equation} \label{eq: var princ intro}
 P(T,f)=\sup\{h_{\mu}(T)+\int_Xf(x)d\mu: \mu\text{ is a }T-\text{invariant probability measure}\},
\end{equation}
where $h_{\mu}(T)$ denotes the metric entropy of $(T,\mu)$. Moreover, whenever the map $T$ satisfies some form of hyperbolicity (for instance, uniform hyperbolicity \cite{Bow75}, partial-hyperbolicity \cite{LOR11,RS18}, non-uniformly hyperbolicity \cite{AMO04,AC,VV}, expansivity and specification \cite{CT16,F77}), it is known that the supremum in the right-hand side of the previous expression is attained by some measure. This measure, called an \emph{equilibrium state}, encodes several properties of the system. For instance, it allows us to determine various fractal dimensions associated to the dynamics \cite{Ba,Pesin}. In particular, thermodynamical formalism and more specifically relation \eqref{eq: var princ intro} provides us a way to choose ``interesting” measures and thus, it provides us candidates to answer our motivating question.

With these motivations in mind, the objective of this paper is twofold: (i) to introduce four notions of topological pressure for non-necessarily continuous semiflows and to study the relation between these notions with the standard one when restricted to the continuous setting; and, (ii) for a class of discontinuous systems known as \emph{impulsive systems}, to obtain an invariance principle.

 \subsection{Impulsive dynamical systems}
 Impulsive dynamical systems may be described by three objects: a continuous semiflow acting on a compact metric space $(X, d)$; a compact set $D\subset X$ where the flow suffer abrupt perturbations; and an impulsive function $I:D\to X $ which describes the perturbations in $D$ and, in general, generates discontinuities in the system. 

This kind of system seem to be an important mathematical model to describe real world phenomena that exhibit sudden changes in their states. For instance, it can be used to give a theoretical characterization of wormholes, also called Einstein-Rosen bridges; to study the population control of some insects with the number of insects and their natural enemies as state variables; to describe a chemical reactor system where the quantities of different chemicals are considered as the states; to model a financial system with two state variables: the amount of money in a market and the saving rates of a central bank. For references to these applications see \cite{LBS89,LWS,KT1994}. We also refer to the introduction of \cite{AC14} and references therein for more applications.

Whereas the study of the \emph{topological properties} of impulsive systems has been extensively studied in the last two decades (see, for instance, \cite{BBCC,LBS89}), an ergodic treatment of this special class of dynamics is still in it’s ``infancy”. For instance, it was only very recently that the existence of invariant measures was established \cite{AC14}. So, as already mentioned, our intent is to provide notions of topological pressure suitable for this context and to establish an invariance principle (see Theorem \ref{theo: var princ}) contributing to the study of thermodynamical formalism of these systems. Some previous work in this direction are \cite{ACV15,JSM19,JSM}, where a variational principle was established for the topological entropy and \cite{ACS17}, where a notion of topological pressure (a priori different from ours) was introduced and a variational principle along with the existence of equilibrium states was obtained. As a consequence of our main result we conclude that, restricted to an appropriate class of impulsive systems, our four notions of topological pressure actually coincide with the one introduced in \cite{ACS17} (see Corollary \ref{cor:relacao-de-tudo}) which also gives us the existence of equilibrium states.

\section{Setting and statements}
Let $(X,d)$ be a compact metric space and $\varphi:\mathbb{R}_0^+\times X\to  X$ a semiflow. This later condition means that $\varphi$ satisfies
\begin{displaymath}
\varphi (0,x)=x \text{ and } \varphi(s+t,x)=\varphi(s,\varphi(t,x))
\end{displaymath}
for all $s,t\in \mathbb{R}_0^+$ and $x\in X$. In what follows we are going to adopt the classical notation $\varphi (t,x)=\varphi_t(x)$.

\subsection{Topological pressure} In this section we introduce the notions of topological pressure that we are going to work with. We start by recalling the classical one.

\subsubsection{Classical definition} 
Let us assume $\varphi:\mathbb{R}_0^+\times X\to  X$ is a continuous semiflow. Given $\varepsilon>0$ and $T>0$, a set $E\subset X$ is said to be \emph{$(\varphi,T,\varepsilon)$-spanning}  if for all $x\in X$, there exists $y\in E$ so that for all $t\in [0,T]$
\begin{displaymath}
d(\varphi_t(x),\varphi_t(y))<\varepsilon.
\end{displaymath}
On the other hand, a set $F\subset X$ is said to be \emph{$(\varphi,T,\varepsilon)$-separated}  if for all $x_1\not=x_2\in F$, there exists  $t\in [0,T]$ so that
\begin{displaymath}
d(\varphi_t(x_1),\varphi_t(x_2))\geq\varepsilon.
\end{displaymath}

Given a continuous potential $f:X\to \mathbb{R}$, we define

\begin{displaymath}
Z_r(\varphi,f,\varepsilon,T) =\inf\left\{\sum_{x\in E}\exp\left(\int_{0}^{T}f(\varphi_s(x))ds\right):E \text{ is  } (\varphi,T,\varepsilon)-\text{spanning}\right\},
\end{displaymath}
and
\begin{displaymath}
Z_s(\varphi,f,\varepsilon,T) =\sup\left\{\sum_{x\in F}\exp\left(\int_{0}^{T}f(\varphi_s(x))ds\right):F \text{ is  } (\varphi,T,\varepsilon)-\text{separated }\right\}.
\end{displaymath}
Then we define the quantities
\begin{displaymath}
 P_s(\varphi,f,\varepsilon)  =\displaystyle\limsup_{T\to \infty}\frac{1}{T}\log Z_s(\varphi,f,\varepsilon,T)
\end{displaymath}
and
\begin{displaymath}
 P_s(\varphi,f) =\lim_{\varepsilon\to 0^+}P_s(\varphi,f,\varepsilon).
\end{displaymath}
Notice that this last limit makes sense once,
\begin{itemize}
\item if $0<\varepsilon_1<\varepsilon_2$, then  $Z_s(\varphi,f,\varepsilon_1,T)\geq Z_s(\varphi,f,\varepsilon_2,T)$;

\item if $0<\varepsilon_1<\varepsilon_2$, then $P_s(\varphi,f,\varepsilon_1)\geq P_s(\varphi,f,\varepsilon_2)$.
\end{itemize}

If instead of considering $Z_s(\varphi,f,\varepsilon,T)$ in the above definitions we consider $Z_r(\varphi,f,\varepsilon,T)$, then we define $P_r(\varphi,f)$. It is a classical result (see for instance \cite{VK16}) that in the above setting these two quantities $P_s(\varphi,f)$ and $P_r(\varphi,f)$ coincide. We define the \textit{topological pressure} of $f$ as
\begin{displaymath}
P(\varphi,f):=P_s(\varphi,f)=P_r(\varphi,f).
\end{displaymath}

\subsubsection{New definitions} \label{sec: new top pressure}
Let $\varphi:\mathbb{R}_0^+\times X\to  X$ be a semiflow (not necessarily continuous). Given $\delta>0$, let us consider the pseudometrics $\hat{d}_\delta^{\varphi}:X\times X\to \mathbb{R}_0^+$ and $\overline{d}_\delta^{\varphi}:X\times X\to \mathbb{R}_0^+$, introduced in \cite{JSM19} and \cite{JSM}, respectively, given by
\begin{displaymath}
\hat{d}_\delta^{\varphi}(x,y)=\inf\{d(\varphi_{s}(x),\varphi_{s}(y)):s\in [0,\delta)\}
\end{displaymath}
and
\begin{displaymath}
\overline{d}_\delta^{\varphi}(x,y)=\inf\{d(\varphi_{s_1}(x),\varphi_{s_2}(y)):s_1,s_2\in [0,\delta)\}.
\end{displaymath}
In order to simplify notation, when there is no risk of ambiguity we write $\hat{d}_\delta$ and $\overline{d}_\delta$ instead of $\hat{d}_\delta^{\varphi}$ and $\overline{d}_\delta^{\varphi}$, respectively.

Given $\varepsilon>0$ and $T>0$, a set $E\subset X$ is said to be \emph{$(\varphi,T,\varepsilon,\delta)$-spanning} with respect to $\overline{d}_\delta$, if for all $x\in X$, there exists $y\in E$ so that for all $t\in [0,T]$
\begin{displaymath}
\overline{d}_\delta(\varphi_t(x),\varphi_t(y))<\varepsilon.
\end{displaymath}
On the other hand, a set $F\subset X$ is said to be \emph{$(\varphi,T,\varepsilon,\delta)$-separated} with respect to $\overline{d} _\delta$, if for all $x_1\not=x_2\in F$ there exists $t\in [0,T]$ so that
\begin{displaymath}
\overline{d}_\delta(\varphi_t(x_1),\varphi_t(x_2))\geq\varepsilon.
\end{displaymath}

Given a continuous potential $f:X\to \mathbb{R}$, we define
\begin{displaymath}
\overline{Z}_r(\varphi,f,\varepsilon,\delta,T) =\inf\left\{\sum_{x\in E}\exp\left(\int_{0}^{T}f(\varphi_s(x))ds\right):E \text{ is } (\varphi,T,\varepsilon,\delta)-\text{spanning w.r.t. $\overline{d} _\delta$}\right\},
\end{displaymath}
and
\begin{displaymath}
\overline{Z}_s(\varphi,f,\varepsilon,\delta,T) =\sup\left\{\sum_{x\in F}\exp\left(\int_{0}^{T}f(\varphi_s(x))ds\right):F \text{ is  } (\varphi,T,\varepsilon,\delta)-\text{separated w.r.t. $\overline{d} _\delta$}\right\}.
\end{displaymath}
Moreover, we define
\begin{align*}
 \overline P_s(\varphi,f,\varepsilon,\delta) & =\displaystyle\limsup_{T\to \infty}\frac{1}{T}\log \overline{Z}_s(\varphi,f,\varepsilon,\delta,T),\\
 \overline P_s(\varphi,f,\delta)& =\lim_{\varepsilon\to 0^+}P_s(\varphi,f,\varepsilon,\delta).
 \end{align*}
The \emph{upper topological pressure} of $f$ is then defined as
\begin{displaymath}
\overline  P_s(\varphi,f) =\lim_{\delta\to 0^+}P_s(\varphi,f,\varepsilon,\delta).
\end{displaymath}

As in the classical setting, the upper topological pressure is well defined, because
\begin{itemize}
\item if $0<\varepsilon_1<\varepsilon_2$, then  $\overline{Z}_s(\varphi,f,\varepsilon_1,\delta,T)\geq \overline{Z}_s(\varphi,\varepsilon_2,\delta,T)$;

\item  if $0<\varepsilon_1<\varepsilon_2$, then $\overline P_s(\varphi,f,\varepsilon_1,\delta)\geq \overline P_s(\varphi,f,\varepsilon_2,\delta)$;

\item  if $0<\delta_1<\delta_2$, then $\overline P_s(\varphi,f,\delta_1)\geq \overline P_s(\varphi,f,\delta_2)$.
\end{itemize}

By considering $\overline{Z}_r(\varphi,f,\varepsilon,\delta,T)$ instead of $\overline{Z}_s(\varphi,f,\varepsilon,\delta,T)$ in the above definitions, we define the \emph{lower topological pressure} of $f$ and denote it by $\overline P_r(\varphi,f)$.

Similarly, by changing the pseudometric $\overline{d}_\delta$ by $\hat{d}_\delta$ we may consider the notions of spanning and separated sets with respect to $\hat{d}_\delta$ and, by considering spanning and separated sets with respect to $\hat{d}_\delta$ instead of spanning and separated sets with respect to $\overline{d}_\delta$, we may consider analogous versions of all the previous notions like $\hat{Z}_r(\varphi,f,\varepsilon,\delta,T)$ and $\hat{Z}_s(\varphi,f,\varepsilon,\delta,T)$ and define
\begin{displaymath}
\hat{P}_r(\varphi,f) \text{ and } \hat{P}_s(\varphi,f).
\end{displaymath}

\begin{remark}\label{rmk1}
If we consider $f$ constant and equal to zero in the above definitions we recover the definitions of topological entropy for semiflows introduced in \cite{JSM19} and \cite{JSM}.
\end{remark}

Our first result show us that, for continuous semiflows, the new notions of topological pressure introduced above coincide with the classical one.
\begin{theorem} \label{theo: pressure continuous}
Let $\varphi:\mathbb{R}_0^+\times X\to  X$ be a continuous semiflow acting on a compact metric space $(X,d)$. Then,
\begin{displaymath}
P(\varphi,f)=\overline{P}_r(\varphi,f)=\overline{P}_s(\varphi,f)=\hat{P}_r(\varphi,f)=\hat{P}_s(\varphi,f).
\end{displaymath}
\end{theorem}
The proof of this result is presented in Section \ref{sec: proof theo cont pressure}.

\subsection{Impulsive semiflows and a variational principle} Let $\varphi:\mathbb{R}^+_0 \times X\to X$ be a continuous semiflow, $D\subset X$ a nonempty compact  set and $I:D \to X$ a continuous map satisfying $I(D)\cap D=\emptyset$. Under these conditions we say that $(X,\varphi,D,I)$ is an \emph{impulsive dynamical system}.

Let $\tau_1:X\to~[0,+\infty]$ be the function given by
\begin{displaymath}
\tau_1(x)=\left\{
\begin{array}{ll}
\inf\left\{t> 0 \colon \varphi_t(x)\in D\right\} ,& \text{if } \varphi_t(x)\in D\text{ for some }t>0;\\
+\infty, & \text{otherwise.}
\end{array}
\right.
\end{displaymath}
Observe that $\tau_1(x)$ gives us the first time the $\varphi$-trajectory of $x$ visits $D$. We then define the \emph{impulsive trajectory} $\gamma_x$ and the subsequent \emph{impulsive times} $\tau_2(x),\tau_3(x),\dots$ (possibly finitely many) of a given point $x\in X$ as follows:

\begin{itemize}[leftmargin=*]
\item If $0\le t<\tau_{1}(x)$ then we set  $\gamma_x(t)=\varphi_t(x)$;

\item We now proceed inductively: assuming that $\gamma_x(t)$ is defined for $t<\tau_{n}(x)$ for some $n\ge 1$, we set
\begin{displaymath}
\gamma_x(\tau_{n}(x))=I(\varphi_{\tau_n(x)-\tau_{n-1}(x)}(\gamma_x({\tau_{n-1}(x)}))).
\end{displaymath}
Then, defining the $(n+1)^{\text{th}}$ impulsive time of $x$ as
\begin{displaymath}
\tau_{n+1}(x)=\tau_n(x)+\tau_1(\gamma_x(\tau_n(x))),
\end{displaymath}
for $\tau_n(x)<t<\tau_{n+1}(x)$, we set
\begin{displaymath}
\gamma_x(t)=\varphi_{t-\tau_n(x)}(\gamma_x(\tau_n(x))).
\end{displaymath}
\end{itemize}
As observed in Remark 1.1 of \cite{AC14}, under the above assumptions, we have that $\sup_{n\ge 1}\,\{\tau_n(x)\}=+\infty$, which guarantees that the impulsive trajectories are defined for all positive times. This allows us to introduce the \emph{impulsive semiflow} $\psi$ associated to the impulsive dynamical system $(X,\varphi, D, I)$ as
\begin{displaymath}
\begin{array}{cccc}
        \psi:  &  \mathbb{R}^+_0 \times X & \longrightarrow &X \\
        & (t,x) & \longmapsto & \gamma_x(t), \\
        \end{array}
\end{displaymath}
where $\gamma_x$ stands for the impulsive trajectory of $x$ determined by $(X,\varphi,D, I)$. It is easy to see that $\varphi$ is indeed a semiflow \cite[Proposition 2.1]{B07}, although not necessarily continuous. Moreover, $\tau_1$ is lower semicontinuous on the set $X\setminus D$ \cite[Theorem~2.7]{C04}. Furthermore, since we are assuming  that $I(D)\cap (D)=\emptyset$ and $I(D)$ is compact, there exists some $\eta>0$ such that $\tau_{n+1}(x)-\tau_n(x)\geq\eta$ for all $x\in X$ and $n\in \mathbb{N}$.

In order to state and prove our results, we need the impulsive system $(X,\varphi,D, I)$ to satisfy some regularity conditions. These conditions were already used in \cite{ACS17,ACV15} and we now recall them.

Given $\xi >0$, let us consider
\begin{equation*}
D_\xi=\bigcup_{x\in D }\{\varphi_t(x) \colon  0<t<\xi\}.
\end{equation*}
We say that $\varphi$ is \emph{$\xi$-regular on $D$} if
\begin{enumerate}
\item $D_t$ is an open set for all $0<t\leq \xi$;
\item if $\varphi_t(x)\in D_{\xi}$ for some  $x\in X$ and $t>0$, then $ \varphi_{s}(x)\in D$  for some $0\le s<t$.
\end{enumerate}
We say that $\varphi$ satisfies a \emph{$\xi$-half-tube condition} on a compact set $A\subset X$ if
\begin{enumerate}
\item $ \varphi_t(x) \in A \,\Rightarrow \,\varphi_{t+s}(x) \notin A$ for all $0 < s < \xi$;
\item $\{\varphi_t(x_1) \colon  0<t\le \xi\}\cap \{\varphi_t(x_2) \colon  0<t\le \xi\}=\emptyset$ for all $x_1,x_2\in A$ with $x_1\neq x_2$;
\item there exists $C>0$ such that, for all $x_1,x_2\in A$ with $x_1\neq x_2$, we have
$$0\leq t<s\leq \xi \quad \Rightarrow \quad d\,(\varphi_t(x_1), \varphi_t(x_2))\leq C\,d\,(\varphi_s(x_1), \varphi_s(x_2)).$$
\end{enumerate}
In what follows, we are going to assume that $\varphi$ satisfies a $\xi$-half-tube condition on the compact sets $D$ and $I(D)$. In particular, the first condition in the definition applied to $A=D$ implies that $\tau_1(x)\geq \xi>0$ for every $ x\in D$.

Given  $\xi>0$  we define
$$X_\xi=X\setminus (D_\xi\cup D).$$
Since $D$ is compact, $I$ is continuous and $I(D)\cap D=\emptyset$, we may choose $\xi$ small enough so that $I(D)\cap D_\xi=\emptyset$. In particular, $X_\xi$ is forward invariant under $\psi$, that is,
\begin{equation}\label{eq: forward inv}
\psi_t(X_\xi) \subseteq X_\xi,\quad \forall t \geq 0.
\end{equation}

We now summarize the properties about impulsive systems that we are going to use in our main result: suppose there exists $\xi_0>0$ such that for all $0<\xi<\xi_0$ we have:
\begin{enumerate}
\item[(C1)]  $I:D \to X$ is Lipschitz  with $\text{Lip}(I)\le 1$ and $I(D)\cap D=\emptyset$;
\item[(C2)] $I(\Omega_\psi \cap D) \subset \Omega_\psi \setminus D$, where $\Omega_\psi$ denotes the set of non-wandering points of $\psi$;
\item[(C3)]  $\varphi$ is $\xi$-regular on   $D$;
\item[(C4)] $\varphi$ satisfies a $\xi$-half-tube condition on both $D$ and $I(D)$;
\item[(C5)] $\tau^*_\xi$ is continuous where $\tau^*_\xi:X_\xi \cup D \to [0,+\infty]$ is given by
\begin{displaymath}
\tau^*_\xi(x)=
\begin{cases}
\tau_1(x), &\text{if $x\in X_\xi$};\\
0, &\text{if $x\in D$}.
\end{cases}
\end{displaymath}
\end{enumerate}

Some comments about our hypothesis are in order. Conditions (C3) and (C4) might, at a first glance, seem very restrictive but they are satisfied, for instance, whenever $\varphi$ is a $C^1$ semiflow on a manifold for which $D$ and $I(D)$ are submanifolds transversal to the flow direction. Regarding conditions (C2) and (C5), according to \cite{AC14}, they are essential to guarantee that the set $\mathcal{M}_\psi(X)$  of $\psi$-invariant measures is nonempty and this fact is prerequisite if one wants to obtain a variational principle, as in our case. Examples of impulsive systems satisfying conditions (C1)-(C5) are given in \cite[Section 7]{ACS17}.

\subsubsection{Admissible potentials} We now recall the class of potentials that we are going to work with. This class was introduced in \cite{ACS17} by refining a class proposed in \cite{F77}. We say that a continuous map $f:X \rightarrow \mathbb{R}$ is an \emph{admissible potential} with respect to the impulsive semiflow $\psi$ (associated to the impulsive system $(X,\varphi,D,I)$) if
\begin{enumerate}
\item $f(x)=f(I(x))$ for every $x\in D$;
\item there exist $K>0$ and $\varepsilon > 0$ such that for every $t>0$  we have
\begin{equation*}
\left|\int_0^t f(\psi_s(x))\,ds - \int_0^t f(\psi_s(y))\,ds\right|< K,
\end{equation*}
whenever
$d\,(\psi_{s}(x), \psi_s(y)) < \varepsilon$ for all $ s \in [0,t]$ such that $\psi_s(x),\psi_s(y)\notin B_\varepsilon(D)$, where $B_\varepsilon(D)$ denote an $\varepsilon$-neighborhood of $D$.
\end{enumerate}
We denote by $\mathcal{A}(\psi)$ the set of all admissible potentials associated to the impulsive semiflow $\psi$. Observe that, for instance, all constant potentials are in $\mathcal{A}(\psi)$.

\subsubsection{Variational Principle} We are now able to state our main result which is a variational principle for impulsive semiflows.

\begin{theorem}\label{theo: var princ}
Let $(X,d)$ be a compact metric space and $\psi$ the semiflow associated to an impulsive dynamical system $(X,\varphi,D,I) $ for which conditions (C1)-(C5) are satisfied. Then, for any admissible potential $f\in  \mathcal{A}(\psi)$ we have
\begin{displaymath}
\overline{P}_r(\psi,f)=\overline{P}_s(\psi,f)=\hat{P}_r(\psi,f)=\hat{P}_s(\psi,f)=\sup_{\mu\in \mathcal{M}_{\psi}(X)} \left\{ h_\mu(\psi_1) + \int f\,d\mu\right\},
\end{displaymath}
where  $\psi_1$ stands for the time one map of the semiflow $\psi$ and $\mathcal{M}_{\psi}(X)$ denotes the space of all $\psi$-invariant probability measures.
\end{theorem}
Observe that by taking $f\equiv 0$ in the previous result we recover \cite[Theorems 3 and 4]{JSM} and, as a consequence of the proof, we recover \cite[Theorem 3]{JSM19}. Moreover, as a subproduct of this result, we get that the four (a priori different) notions of topological pressure introduced in Section \ref{sec: new top pressure} actually coincide for the class of impulsive semiflows considered in the statement. Furthermore, by combining Theorem \ref{theo: var princ} with \cite[Theorem C]{ACS17} we get the following.
\begin{corollary}\label{cor:relacao-de-tudo}
Let $P^\tau(\psi,f)$ denote the topological pressure introduced in \cite{ACS17}. Then,
\begin{displaymath}
P^\tau(\psi,f)=\overline{P}_r(\psi,f)=\overline{P}_s(\psi,f)=\hat{P}_r(\psi,f)=\hat{P}_s(\psi,f)
\end{displaymath}
for any admissible potential $f\in  \mathcal{A}(\psi)$.
\end{corollary}
In particular, this corollary combined with \cite[Theorem A]{ACS17} implies that whenever $\psi$ is \emph{positively expansive} and has the \emph{periodic specification property} in $\Omega_\psi \setminus D$ (see \cite{ACS17} for the precise definitions of these concepts), for any $f\in \mathcal{A}(\psi)$ there exists an \emph{equilibrium state}. That is, there exists a measure $\mu\in \mathcal{M}_{\psi}(X)$ realizing the supremum at the right-hand side of the equality given in Theorem \ref{theo: var princ}. Moreover, if $\text{dim}(X)<\infty$ and there exists $k>0$ so that $\#I^{-1}(\{y\})\leq k$ for every $y\in I(D)$, then the equilibrium state is actually unique.

\begin{remark}
The importance of Corollary \ref{cor:relacao-de-tudo} stems from the fact that, since we have five possible definitions of topological pressure for a regular impulsive semiflow and all of them give rise to the same quantity, we can choose the one that best suits the problem we are interested in.
\end{remark}
The proof of Theorem \ref{theo: var princ} is presented in Section \ref{sec: proof var princ}.

\section{Proof of Theorem \ref{theo: pressure continuous}} \label{sec: proof theo cont pressure}

In order to prove Theorem \ref{theo: pressure continuous}, we need an auxiliary result. This result also justifies the names \emph{upper} and \emph{lower} topological pressures.

\begin{lemma}\label{lemma: lower upper}
Let $\varphi $ be a semiflow acting on $X$ and $f:X\to \mathbb{R}$ be a continuous potential. Then
\begin{displaymath}
\overline{P}_r(\varphi,f)\leq \overline{P}_s(\varphi,f) \text{ and } \hat{P}_r(\varphi,f)\leq \hat{P}_s(\varphi,f).
\end{displaymath}
\end{lemma}
\begin{proof}
We prove only the assertion about $\overline{P}_r$ and $\overline{P}_s$. The other one is similar. We follow the classical approach used to prove that the topological entropy defined via spanning sets is smaller than or equal to the topological entropy defined via separated sets (see \cite{JSM19,VK16,W82}).
Fix $\delta>0$, $\varepsilon >0$ and $T>0$ and consider the partial order given by set inclusion on the set of all $(\varphi, T,\varepsilon, \delta)$-separated sets with respect to $\overline{d}_\delta$. Now, since the union of a partially ordered family of separated sets is still a separated set, it follows from Zorn's Lemma that there exists a maximal $(\varphi, T,\varepsilon, \delta)$-separated set $E$. We observe that this set $E$ is also a $(\varphi, T,\varepsilon, \delta)$-spanning set. Indeed, suppose that $E$ is not a $(\varphi,T,\varepsilon,\delta)$-spanning set. Then, there exists $x\in X$ such that for all $y\in E$, there exists $t\in[0,T]$ such that
\begin{displaymath}
\overline{d}_\delta(\varphi_t(x),\varphi_t(y))\geq \varepsilon.
\end{displaymath}
In particular, $E\cup\{x\}$ is a $(\varphi,T,\varepsilon,\delta)$-separated set contradicting $E$'s maximality. Consequently, for any $f\in \mathcal{A}(\psi) $ we have that
\begin{displaymath}
\overline{Z}_r(\varphi,f,\varepsilon,\delta,T) \leq \overline{Z}_s(\varphi,f,\varepsilon,\delta,T) .
\end{displaymath}
Thus, by taking logarithm on both sides of the inequality, dividing by $T$ and letting $T\to+\infty$, $\varepsilon\to 0^+$ and $\delta \to 0^+$ it follows that
\begin{displaymath}
\overline{P}_r(\varphi,f)\leq \overline{P}_s(\varphi,f)
\end{displaymath}
as claimed.
\end{proof}

\begin{proof}[Proof of Theorem \ref{theo: pressure continuous}]
Once again we prove only the assertions involving $\overline{P}_r$ and $\overline{P}_s$ since the ones about $\hat{P}_r$ and $\hat{P}_s$ are similar.

Fix $\delta>0$ and $\varepsilon>0$. We start observing that, for $T>0$, if $F$ is a $(\varphi,T,\varepsilon,\delta)$-separated set with respect to $\overline{d}_\delta$ then $F$ is a $(\varphi,T,\varepsilon)$-separated set. In fact, this follows easily from the fact that $\overline{d}_\delta(x,y)\leq d(x,y)$ for every $x,y\in X$. Thus,
\begin{displaymath}
\overline{Z}_s(\varphi,f,\varepsilon,\delta,T)
\leq Z_s(\varphi,f,\varepsilon, T),
\end{displaymath}
which gives us
\begin{equation}\label{eq: aux ps}
\overline{P}_s(\varphi,f)\leq  P_s(\varphi,f).
\end{equation}
In what follows we are going to show that
\begin{align}\label{eq:4.1}
\overline P_r(\varphi,f)\geq  P_r(\varphi,f).
\end{align}

By the continuity of $\varphi$ and the compactness of $X$, given $\alpha >0$ there exists $\beta=\beta(\alpha)>0$ so that for all $x\in X$ and $t\geq 0$, we have
\begin{equation}\label{eq:1}
    u\in [t,t+\beta)\Rightarrow d(\varphi_t(x),\varphi_u(x))<\frac{\alpha}{4}.
\end{equation}
 Now, let $\varepsilon >0$, $\delta>0$ and $T>0$ be numbers satisfying $\delta<\beta$ and $\varepsilon<\frac{\alpha}{2}$. We observe that if $E$ is a $(\varphi,T,\varepsilon,\delta)$-spanning set with respect to $\overline{d}_\delta$, then $E$ is a $(\varphi,T,\alpha)$-spanning set. In fact, suppose this is not the case. Then, there exists $x\in X$ such that for all $y\in E$,  there exists $t_0\in [0,T] $ with $d(\varphi_{t_0},(x),\varphi_{t_0}(y))\geq \alpha$. Using the fact that
\begin{displaymath}
\alpha\leq d(\varphi_{t_0}(x),\varphi_{t_0}(y))\leq d(\varphi_{t_0}(x),\varphi_{u}(y))+d(\varphi_{u}(x),\varphi_{s}(y))+d(\varphi_{s}(y),\varphi_{t_0}(y)),
\end{displaymath}
 and \eqref{eq:1} we get that for all $u,s\in [ t,t+\beta)$,
\begin{displaymath}
d(\varphi_u(x),\varphi_s(y))>\frac{\alpha}{2}.
\end{displaymath}
By the choice of $\delta$ and $\varepsilon$, this last inequality implies that
\begin{displaymath}
\overline{d}_{\delta}(\varphi_{t_0}(x),\varphi_{t_0}(y))>\varepsilon,
\end{displaymath}
which is a contradiction. So, any $(\varphi,T,\varepsilon,\delta)$-spanning set with respect to $\overline{d}_\delta$ is a $(\varphi,T,\alpha)$-spanning set, which ensures that
\begin{displaymath}
\overline{Z}_r(\varphi,f,\varepsilon,\delta,T)\geq Z_r(\varphi,f,\alpha,T)
\end{displaymath}
and completes the proof of \eqref{eq:4.1}. Thus, combining \eqref{eq:4.1}, Lemma \ref{lemma: lower upper} and \eqref{eq: aux ps} we get that
\begin{displaymath}
P_r(\varphi,f)\leq \overline P_r(\varphi,f)\leq \overline P_s(\varphi,f)\leq P_s(\varphi,f).
\end{displaymath}
Then, since for a continuous semiflow we have that $P_s(\varphi,f)=P_r(\varphi,f)=P(\varphi,f)$, we conclude that
\begin{displaymath}
P(\varphi,f)= \overline P_r(\varphi,f)= \overline P_s(\varphi,f)
\end{displaymath}
as claimed.
\end{proof}

\section{Proof of the Variational Principle} \label{sec: proof var princ}
In this section we prove the variational principle. With this purpuse in mind, we start recalling a useful construction introduced in \cite{AC14}. We follow the presentation given in \cite{ACS17}.
\subsection{Quotient dynamics} \label{sec: quotient}
Given an impulsive dynamical system $(X,\varphi,D,I)$, let us consider the quotient space $\widetilde{X}=X/\sim$, where $\sim$ is the equivalence relation on $X$ given by
\begin{displaymath}
 x\sim y\quad \Leftrightarrow \quad x=y, \quad y=I(x),\quad x=I(y)\quad \text{or}\quad I(x)=I(y).
\end{displaymath}
We denote by $\tilde{x}$ the equivalence class of $x\in X$. Let us consider $\widetilde{X}$ endowed with the quotient topology and let $\pi:X\to \widetilde{X}$ be the natural projection. A simple observation (see \cite[Lemma 4.1]{AC14}) is that this space is a metrizable space and, moreover, a metric $\tilde{d}$ on $\widetilde{X}$ that induces the quotient topology is given by
\begin{displaymath}
 \tilde{d} (\tilde{x},\tilde{y}) = \inf  \{d(p_1, q_1)+ d(p_2, q_2)+\cdots +d(p_n,q_n)\},
\end{displaymath}
where $p_1, q_1, \dots, p_n, q_n$ is any chain of points in $X$ such that $ p_1 \sim x$, $q_1 \sim p_2$, $q_2 \sim p_3$, ... $q_n \sim y$. In particular, we have for every $x,y\in X$,
\begin{displaymath}
\tilde{d}(\tilde{x},\tilde{y}) \leq d(x,y).
\end{displaymath}
Moreover, whenever $I$ does not expand distances we can get a bi-Lipschitz relation between $d$ and $\tilde{d}$ in the following sense: if $\text{Lip}(I)\leq 1$, then for all $\tilde{x}, \tilde{y} \in \pi(X)$ there exist $p, q \in X$ such that
$p \sim x$, $  q \sim y$ and $ d(p,q) \leq 2\,\tilde{d}(\tilde{x}, \tilde{y})$ (see \cite[Lemma 4.1]{ACV15}).

Our objective now is to construct a dynamics on a subset of $\pi(X)$ that is induced by $\psi$. The advantage of considering such induced dynamics is that it will be a continuous one defined on a compact metric space for which there are plenty of results from Ergodic Theory available. The general idea then is to pull this results back to the original dynamics (which is not necessarily continuous) via (semi)conjugacy. This idea has already been explored in some previous works that motivated our own (see for instance, \cite{AC14,ACS17,ACV15,JSM19}). In order to put this idea to work, take $\xi>0$ such that conditions (C1)-(C5) hold. Since $I(D)\cap D = \emptyset$, each point in the set $X_\xi=X\setminus (D_\xi\cup D)$ has a representative of its equivalence class in $X\setminus D_\xi $. This implies that
\begin{displaymath}
\pi(X_\xi)=\pi(X\setminus D_\xi),
\end{displaymath}
which, by condition (C3), is a compact set. Moreover, as we are assuming that $\varphi$ satisfies a $\xi$-half-tube condition, it follows that $X_\xi$ is $\psi_t$-invariant (recall \eqref{eq: forward inv}) and, since there is no risk of confusion, we will still denote the restriction of $\psi$ to $X_\xi$ by $\psi$.

Given $x,y\in X_\xi$ we have $x\sim y$ if and only if $x=y$. In particular, $\pi_{|X_\xi}$ induces a continuous bijection from $X_\xi$ onto the set
\begin{displaymath}
\widetilde{X_\xi}=\pi(X_\xi)
\end{displaymath}
that we are going to denote by $H$. This map $H$ allows us to introduce a semiflow
$\tilde{\psi}$ on $ \widetilde{X_\xi} $
given by
\begin{displaymath}
\tilde{\psi}(t,\tilde x)=H \circ\psi(t,x),
\end{displaymath}
for all $x\in X_\xi$ and  $t\geq 0$. Since the impulsive semiflow $\psi$ satisfies conditions (C1) and (C5), it follows from \cite[Lemma 4.2]{ACV15} that the semiflow $\tilde{\psi}$ is continuous. In particular, this is the induced dynamics that we were looking for. Moreover, from the definition of the map $H$, we have a  semiconjugacy between the semiflows $\psi$ and $\tilde{\psi}$. That is,
\begin{displaymath}
\tilde{\psi}_t( H (x) )= H(\psi_t (x))
\end{displaymath}
for every $x\in X_\xi$ and $t\geq 0$. The main features of these constructions can be summarized as follows.

\begin{proposition}\label{prop: conjugacy}
Let $(X,d)$ be a compact metric space and $\psi$ be  the semiflow associated to an impulsive dynamical system $(X,\varphi,D,I) $ for which conditions (C1)-(C5) are satisfied. Then, there exists a compact metric space $(\widetilde{X_\xi},\tilde{d})$, a continuous semiflow $\tilde{\psi}:\mathbb{R}_0^+\times \widetilde{X_\xi}\to \widetilde{X_\xi}$ and a uniformly continuous bijection $H:X_\xi\to \widetilde{X_\xi}$ so that for all $t\geq 0$
\begin{displaymath}
\tilde{\psi}_t\circ H=H\circ \psi_t.
\end{displaymath}

\end{proposition}

\subsection{Lemmata}\label{subsec: lemmata}
In this subsection we present several useful auxiliary results. We retain all the notation already introduced and denote the pseudometrics $\overline{d}^{\psi}_\delta$ and $\hat{d}^{\psi}_\delta$ introduced in Section \ref{sec: new top pressure} simply by $\overline{d}_\delta$ and $\hat{d}_\delta$, respectively. Moreover, we fix $\xi$ sufficiently small so that $0<\xi<\frac{\xi_0}{4}$, where $\xi_0$ is given in properties (C1)-(C5) and, for $f\in \mathcal{A}(\psi)$, we consider
\begin{displaymath}
\tilde{f}:=f\circ H^{-1}=f\circ (\pi_{|X_\xi})^{-1}.
\end{displaymath}
We start with a very simple and general observation relating the notions of topological pressure defined using $\overline{d}_\delta$ and $\hat{d}_\delta$.
\begin{lemma}\label{lemma: Pover x Phat}
\begin{displaymath}
\overline{P}_r(\psi,f)\leq \hat{P}_r(\psi,f) \text{ and } \overline{P}_s(\psi,f)\leq \hat{P}_s(\psi,f).
\end{displaymath}
\end{lemma}
\begin{proof}
It follows easily since $\overline{d}_\delta(x,y)\leq \hat{d}_\delta(x,y)$ for every $x,y\in X$.
\end{proof}

From the observations of Section \ref{sec: quotient}, we know that $X_\xi$ is $\psi_t$-invariant. Our next two results relate the pressures of $\psi$ and $\psi_{|X_\xi}$.

\begin{lemma}\label{lemma: Pover restric leq Pover psi}
\begin{displaymath}
\overline{P}_r(\psi_{|X_\xi},f)\leq \overline{P}_r(\psi,f).
\end{displaymath}
\end{lemma}

\begin{proof}
Given $\varepsilon>0$, $\delta>0$ and $T >0$, let $F\subset X$ be a $(\psi, T,\varepsilon,\delta)$-spanning set with respect to $\overline{d}_\delta$ and consider
\begin{displaymath}
F':=\left\{\psi_{t(y)}(y);\; y\in F\right\},
\end{displaymath}
where $t(y):=\inf \{t\geq 0; \psi_{t(y)}(y)\in X_\xi\}$. From the proof of \cite[Proposition 1]{JSM} we know that whenever $\varepsilon$ and $\delta$ are sufficiently small and $T> 2\delta$, $F'$ is a $(\psi_{|X_\xi}, T-2\delta,\varepsilon,3\delta)$-spanning set with respect to $\overline{d}_\delta$.

Now, since $f \in \mathcal{A}(\psi)$, there exist $\rho>0$ and $K>0$ such that for every $t>0$ we have
\begin{equation}\label{eq: estimative integral f}
\left|\int_0^t f(\psi_s(x))\,ds - \int_0^t f(\psi_s(y))\,ds\right|< K,
\end{equation}
whenever $d(\psi_{s}(x), \psi_s(y)) < \rho$ for all $ s \in [0,t]$ such that $\psi_s(x),\psi_s(y)\notin B_{\rho}(D)$. Moreover, using continuity of $\varphi$ and compactness of $X$ (see \cite[Lemma 2.1]{ACV15}) we can find $\alpha>0$ such that $d(\varphi_s(x),\varphi_u(x))<\rho$ for all $x\in X$ and  $s,u\geq 0$ with $|s-u|<\alpha$. Therefore, choosing  $0<\xi<\alpha$, we have
\begin{displaymath}
d(\psi_{s}(x), \psi_s(\psi_\xi (x))) =d(\varphi_{s}(x), \varphi_s(\varphi_\xi (x))) < \rho
\end{displaymath}
for all $s \in [0,t]$ such that $\psi_s(x),\psi_s(\psi_\xi(x))\notin B_{\rho}(D)$. Consequently, noting that $t(x)\leq \xi$ for every $x \in F\cap (D\cup D_\xi)$, we get that
\begin{displaymath}
\sum_{x \in F'\setminus F} \exp\left(\int_0^t f(\psi_s(x))ds\right)=
\sum_{y \in  F\cap (D\cup D_\xi)} \exp\left(\int_0^t f(\psi_s(\psi_{t(y)}(y)))ds\right)
\end{displaymath}
is equal to
\begin{displaymath}
 \sum_{y \in F\cap (D\cup D_\xi)}\exp\left( \int_0^t f(\psi_s(\psi_{t(y)}(y)))ds-\int_0^t f(\psi_s(y))ds+\int_0^t f(\psi_s(y))ds \right),
\end{displaymath}
which is smaller than or equal to
\begin{displaymath}
\sum_{y \in F\cap (D\cup D_\xi)} \exp\left(\left| \int_0^t f(\psi_s(\psi_{t(y)}(y)))ds -\int_0^t f(\psi_s(y))ds \right|\right) \exp\left(\int_0^t f(\psi_s(y))ds\right),
\end{displaymath}
which by its turn is smaller than or equal to
\begin{displaymath}
  e^K \sum_{ x \in F\cap (D\cup D_\xi)}\exp\left(\int_0^t f(\psi_s(y))ds\right).
\end{displaymath}

Thus, as
\begin{displaymath}
\sum_{x \in F} \exp\left({\int_0^t f(\psi_s(x))ds}\right) = \sum_{x \in F\cap (D\cup D_\xi)} \exp\left({\int_0^t f(\psi_s(x))ds}\right) + \sum_{x \in F\cap X_\xi} \exp\left({\int_0^t f(\psi_s())ds}\right)
\end{displaymath}
and, similarly,
\begin{displaymath}
\sum_{x \in F'} \exp\left({\int_0^t f(\psi_s(x))ds}\right) = \sum_{x \in F'\setminus F} \exp\left({\int_0^t f(\psi_s(x))ds}\right) + \sum_{x \in F\cap X_\xi} \exp\left({\int_0^t f(\psi_s(x))ds}\right)
\end{displaymath}
and using the previous observation it follows that
\begin{displaymath}
\sum_{x \in F'} \exp\left({\int_0^t f(\psi_s(x))ds}\right) \leq e^K  \sum_{x \in F} \exp\left({\int_0^t f(\psi_s(y))ds}\right).
\end{displaymath}
Therefore,
\begin{displaymath}
\overline{Z}_r(\psi_{|X_\xi},f,\varepsilon,\delta,T-2\delta)\leq e^K \overline{Z}_r(\psi,f,\varepsilon,\delta,T)
\end{displaymath}
for every $T>2\delta$ and $\varepsilon$ and $\delta$ sufficiently small. Consequently, by taking logarithm on both sides, dividing by $T$ and making $T\to+\infty$, $\varepsilon\to 0^+$ and $\delta \to 0^+$ it follows that
\begin{displaymath}
\overline{P}_r(\psi_{|X_\xi},f)\leq \overline{P}_r(\psi,f).
\end{displaymath}
\end{proof}

\begin{lemma}\label{lem: Ps psi and Ps psi restrita}
\begin{displaymath}
\hat{P}_s(\psi_{|X_\xi},f) = \hat{P}_s(\psi,f).
\end{displaymath}
\end{lemma}
\begin{proof}

Since $X_\xi\subset X$, it follows that $\hat{P}_s(\psi_{|X_\xi},f)\leq\hat{P}_s(\psi,f)$. We now prove that the converse inequality is also true.

Given $T>0$, $\varepsilon>0$ and $\delta>0$, let $E\subset X$ be a finite $(\psi, T,\varepsilon,\delta)$-separated set with respect to $\hat{d}_\delta$ and consider
\begin{displaymath}
E_\xi=E\cap X_\xi,\; E_D=E\cap D \text{ and } N_\xi=E\cap  D_\xi.
\end{displaymath}
Since $X_\xi$ is $\psi_t$-invariant it follows easily that $E_\xi\subset X_\xi$ is a $(\psi_{|X_\xi}, T,\varepsilon,\delta)$-separated set with respect to $\hat{d}_\delta$. In particular,
\begin{equation}\label{eq: estimative Exi}
 \sum_{x \in E_\xi} \exp\left({\int_0^t f(\psi_s(x))ds}\right) \leq \hat{Z}_s(\psi_{|X_\xi},f,\varepsilon,\delta,T).
\end{equation}

Similarly, since for any $x\in D$ we have $\psi_t(x)\in X_\xi$ for every $t>0$,  $\psi_t(E_D)$ is a $(\psi_{|X_\xi}, T,\varepsilon,\frac{\delta}{2})$-separated set with respect to $\hat{d}_\delta$ for every $t>0$ sufficiently small (recall the definition of $\hat{d}_\delta$ and that $\text{Lip}(I)\leq 1$ (condition (C1))). In particular,
\begin{equation}\label{eq: estimative ED}
 \sum_{x \in E_D} \exp\left({\int_0^t f(\psi_s(x))ds}\right) \leq \hat{Z}_s(\psi_{|X_\xi},f,\varepsilon,\frac{\delta}{2},T).
\end{equation}

We now claim that there exists a constant $C=C(\varepsilon)>0$, independent of $T$ and $E$, so that
\begin{equation}\label{eq: estimative Nxi}
 \sum_{x \in N_\xi} \exp\left({\int_0^t f(\psi_s(x))ds}\right) \leq C\hat{Z}_s(\psi_{|X_\xi},f,\varepsilon,\delta,T).
\end{equation}
By the compactness of  $\overline{D_\xi}$  and the continuity of  $\varphi_t$, there exists $r>0$ such that for any $x,y\in \overline{D_\xi}$ satisfying $d(x,y)<2r$ we have that
\begin{displaymath}
d(\varphi_t(x),\varphi_t(y))<\varepsilon, \text{ for all } t\in[0,\xi_0-\xi]
\end{displaymath}
where $\xi_0$ is given in properties (C1)-(C5). By compactness, there exists $\{z_k\}_{k=1}^n$, with $n$ depending on $\varepsilon$ but not on $T$ nor on $E$, such that
\begin{displaymath}
\overline{D_\xi}\subset \bigcup_{k=1}^n B(z_k,r).
\end{displaymath}
Thus, considering
\begin{displaymath}
N_\xi^k=N_\xi \cap B(z_k,r)
\end{displaymath}
for every $k=1,2,\ldots,n$ we have that $N_\xi=\cup_{k=1}^nN_\xi^k$ and, moreover, for any $x,y \in N_\xi^k$,
\begin{displaymath}
d(\varphi_t(x),\varphi_t(y))<\varepsilon, \text{ for every } t\in[0,\xi_0-\xi].
\end{displaymath}
In particular, since ${\varphi_t}_{|D_\xi}={\psi_t}_{|D_\xi}$ for every $t\in[0,\xi_0-\xi]$, we get that for any $x,y \in N_\xi^k$,
\begin{displaymath}
d(\psi_t(x),\psi_t(y))<\varepsilon, \text{ for every } t\in[0,\xi_0-\xi].
\end{displaymath}
Thus, assuming $\delta<\frac{\xi_0-\xi}{2}$ we get that
\begin{displaymath}
\hat{d}_{\delta}(\psi_t(x),\psi_t(y))<\varepsilon, \mbox{ for every }t \in \left[0,\frac{\xi_0-\xi}{2}\right].
\end{displaymath}
Therefore, since $E$ is $(\psi,T,\varepsilon,\delta)$-separated with respect to $\hat{d}_\delta$, for every $x, y\in N_\xi^k$ with $x\neq y$, there exists $t\in[\frac{\xi_0-\xi}{2},T]$ such that
\begin{displaymath}
\hat{d}_{\delta}(\psi_t(x),\psi_t(y))\geq\varepsilon.
\end{displaymath}
Consequently, recalling that $\frac{\xi_0-\xi}{2}>\xi$ and $\varphi_{\xi}(\overline{D_{\xi}})\subset X_{\xi}$, we get that $\psi_\xi(N_\xi^k)$ is a $(\psi|_{X_{\xi}},T-\xi,\varepsilon,\delta)$-separated set with respect to $\hat{d}_\delta$. Now, using inequality \eqref{eq: estimative integral f} and the comments that follow it, we get that, for every $f\in \mathcal{A}(\psi)$ and $k\in \{1,2,\ldots,n\}$,
\begin{displaymath}
\sum_{x \in N_\xi^k} \exp\left(\int_0^t f(\psi_s(x))ds\right)
\end{displaymath}
is equal to
\begin{displaymath}
 \sum_{x \in N_\xi^k}\exp\left(\int_0^t f(\psi_s(x))ds - \int_0^t f(\psi_s(\psi_{\xi}(x)))ds + \int_0^t f(\psi_s(\psi_{\xi}(x)))ds \right),
\end{displaymath}
which is smaller than or equal to
\begin{displaymath}
\sum_{x \in N_\xi^k} \exp\left(\left|\int_0^t f(\psi_s(x))ds - \int_0^t f(\psi_s(\psi_{\xi}(x)))ds \right|\right) \exp\left(\int_0^t f(\psi_s(\psi_{\xi}(x)))ds\right),
\end{displaymath}
which, by its turn, is smaller than or equal to
\begin{displaymath}
  e^K \sum_{ x \in N_\xi^k}\exp\left(\int_0^t f(\psi_s(\psi_{\xi}(x)))ds\right),
\end{displaymath}
which, finally, is equal to
\begin{displaymath}
e^K \sum_{y \in \psi_\xi(N_\xi^k)}\exp\left(\int_0^t f(\psi_s(y))ds\right).
\end{displaymath}
This fact combined with the previous observation that $\psi_\xi(N_\xi^k)$ is a $(\psi|_{X_{\xi}},T-\xi,\varepsilon,\delta)$-separated set with respect to $\hat{d}_\delta$, implies that
\begin{displaymath}
\sum_{x \in N_\xi^k} \exp\left(\int_0^t f(\psi_s(x))ds\right)\leq e^K\hat{Z}_s(\psi_{|X_\xi},f,\varepsilon,\delta,T).
\end{displaymath}
Consequently, recalling that $N_\xi=\cup_{k=1}^nN_\xi^k$,
\begin{displaymath}
\begin{split}
\sum_{x \in N_\xi} \exp\left(\int_0^t f(\psi_s(x))ds\right)&\leq\sum_{k=1}^n \left( \sum_{x \in N_\xi^k} \exp\left(\int_0^t f(\psi_s(x))ds\right)\right)\\
&\leq n e^K\hat{Z}_s(\psi_{|X_\xi},f,\varepsilon,\delta,T).
\end{split}
\end{displaymath}
Thus, taking $C=ne^K$ we complete the proof of the claim given in \eqref{eq: estimative Nxi}.

Now, using \eqref{eq: estimative Exi}, \eqref{eq: estimative ED} and \eqref{eq: estimative Nxi}, the monotonicity of $\hat{Z}_s(\psi_{|X_\xi},f,\varepsilon,\delta,T)$ with respect to $\delta$ and recalling that $E=E_\xi\cup E_D\cup N_\xi$, it follows that
\begin{equation*}
 \sum_{x \in E} \exp\left({\int_0^t f(\psi_s(x))ds}\right) \leq (C+2)\hat{Z}_s(\psi_{|X_\xi},f,\varepsilon,\frac{\delta}{2},T),
\end{equation*}
where $C$ is a constant that depends neither on $T$ nor on $E$. So, since this inequality holds true for any finite $(\psi, T,\varepsilon,\delta)$-separated set $E$, we get that
\begin{displaymath}
\hat{Z}_s(\psi,f,\varepsilon,\delta,T) \leq (C+2)\hat{Z}_s(\psi_{|X_\xi},f,\varepsilon,\frac{\delta}{2},T).
\end{displaymath}
Therefore, taking logarithm on both sides of the inequality, dividing by $T$ and making $T\to+\infty$, $\varepsilon\to 0^+$ and $\delta \to 0^+$ it follows that
\begin{displaymath}
\hat{P}_s(\psi,f)\leq \hat{P}_s(\psi_{|X_\xi},f),
\end{displaymath}
which combined with the initial observation concludes the proof of the lemma.
\end{proof}

In what follows we relate the topological pressures of $(\tilde{\psi},\tilde{f})$ and $(\psi_{|X_\xi},f)$. Recall $\tilde{\psi}$ defined  in Section \ref{sec: quotient} and $\tilde{f}=f\circ H^{-1}$. Moreover, given $\delta>0$ we consider the pseudometrics
\begin{displaymath}
\hat{d}_\delta^{\tilde{\psi}}(x,y)=\inf\{\tilde{d}(\tilde{\psi}_{s}(x),\tilde{\psi}_{s}(y)):s\in [0,\delta)\}
\end{displaymath}
and
\begin{displaymath}
\overline{d}_\delta^{\tilde{\psi}}(x,y)=\inf\{\tilde{d}(\tilde{\psi}_{s_1}(x),\tilde{\psi}_{s_2}(y)):s_1,s_2\in [0,\delta)\}
\end{displaymath}
on $\widetilde{X_\xi}$ induced by $\tilde{d}$ and $\tilde{\psi}$. These are the versions of $\hat{d}_\delta$ and $\overline{d}_\delta$ for $(\widetilde{X_\xi},\tilde{d}, \tilde{\psi})$. In particular, this allows us to consider $\overline{P}_r(\tilde{\psi},\tilde{f})$, $\overline{P}_s(\tilde{\psi},\tilde{f})$, $\hat{P}_r(\tilde{\psi},\tilde{f})$ and $\hat{P}_s(\tilde{\psi},\tilde{f})$.

\begin{lemma}\label{lemma: ineq conjugacy}
\begin{displaymath}
\overline{P}_s(\tilde{\psi},\tilde{f})\leq \overline{P}_s(\psi_{|X_\xi},f) \text{ and }\overline{P}_r(\tilde{\psi},\tilde{f})\leq \overline{P}_r(\psi_{|X_\xi},f).
\end{displaymath}
Similarly,
\begin{displaymath}
\hat{P}_s(\tilde{\psi},\tilde{f})\leq \hat{P}_s(\psi_{|X_\xi},f) \text{ and } \hat{P}_r(\tilde{\psi},\tilde{f})\leq \hat{P}_r(\psi_{|X_\xi},f).
\end{displaymath}
\end{lemma}
\begin{proof}
Since the arguments to prove these four inequalities are very similar we only prove the first one. Recall that by the constructions presented in Section \ref{sec: quotient} there exists a uniformly continuous bijection $H:X_\xi\to \widetilde{X_\xi}$ so that for all $t\geq 0$
\begin{equation}\label{eq: conjugacy proof}
\tilde{\psi}_t\circ H=H\circ \psi_t.
\end{equation}

Let $\varepsilon>0$. By the uniform continuity of $H$, there exists $\beta=\beta(\varepsilon)>0$ so that
\begin{equation}\label{eq:2}
  d(x,y)<\beta\Rightarrow \tilde{d}(H(x),H(y))<\varepsilon.
  \end{equation}
Now take $\delta>0$, $T>0$ and let $\tilde{F}\subset \widetilde{X_\xi}$ be a $(\tilde{\psi}, \delta,\varepsilon, T)$-separated set with respect to $\overline{d}_\delta^{\tilde{\psi}}$ and notice that
$F=H^{-1}(\tilde{F})\subset X_\xi$ is a $(\psi,\delta,\beta,T)$-separated set with respect to $\overline{d}_\delta$. In fact, given $x,y\in F$ with $x\neq y$, we have that $H(x)\neq H(y)$ and, moreover, since these points are $(\tilde{\psi}, \delta,\varepsilon, T)$-separated with respect to $\overline{d}_\delta^{\tilde{\psi}}$, there exists $t\in[0,T]$ for which
\begin{displaymath}
\overline{d}_\delta^{\tilde{\psi}}(\tilde{\psi}_t(H(x)),\tilde{\psi}_t(H(y)))\geq \varepsilon.
\end{displaymath}
Thus, recalling \eqref{eq: conjugacy proof}, it follows that
\begin{displaymath}
\overline{d}_\delta^{\tilde{\psi}}(H(\psi_t(x)),H(\psi_t(y)))\geq \varepsilon.
\end{displaymath}
Consequently, invoking \eqref{eq:2}, we get that
\begin{displaymath}
\overline{d}_\delta(\psi_t(x),\psi_t(y))\geq\beta.
\end{displaymath}
This proves that $F\subset X_\xi$ is, indeed, a $(\psi,\delta,\beta,T)$-separated set with respect to $\overline{d}_\delta$. Moreover, since $H$ is a bijection and recalling once again that $\tilde{\psi}_t\circ H=H\circ \psi_t$ and $\tilde{f}=f\circ H^{-1}$,
\begin{displaymath}
\begin{split}
\sum_{y\in \tilde{F}}\exp\left(\int_{0}^{T}\tilde{f}(\tilde{\psi}_s(y))ds\right) &=\sum_{x\in H^{-1}( \tilde{F})}\exp\left(\int_{0}^{T}\tilde{f}(\tilde{\psi}_s(H(x)))ds\right) \\
&=\sum_{x\in F}\exp\left(\int_{0}^{T}f(\psi_s(x)ds\right).\\
\end{split}
\end{displaymath}
Combining these two observations we obtain that
\begin{displaymath}
\overline{Z}_s(\tilde{\psi},\delta,\varepsilon,T)\leq  \overline{Z}_s(\psi_{|X_\xi},\delta,\beta,T).
\end{displaymath}
Finally, as $\lim_{\varepsilon\to0^+}\beta(\varepsilon)=0$, taking logarithms, dividing by $T$ and taking the appropriate limits we conclude the proof of the first inequality.
\end{proof}

\begin{lemma}\label{lem: Ps psi tilde and Ps psi restrita}
\begin{displaymath}
P(\tilde{\psi},\tilde{f})= \hat{P}_s(\psi_{|X_\xi},f).
\end{displaymath}
\end{lemma}
\begin{proof}
The fact that $P(\tilde{\psi},\tilde{f})=\hat{P}_s(\tilde{\psi},\tilde{f})\leq \hat{P}_s(\psi_{|X_\xi},f)$ follows from Theorem \ref{theo: pressure continuous} and Lemma \ref{lemma: ineq conjugacy}.

In order to prove the converse inequality, take $\delta>0$, $T>0$, $0<\varepsilon<\delta$ and let $E\subset X_\xi$ be a $(\psi_{|X_\xi}, T,\varepsilon,\delta)$-separated set with respect to $\hat{d}_\delta$. It follows from the proof of \cite[Lemma 5.6]{JSM19} that $\pi (E)$ is a $(\tilde{\psi},T,\varepsilon)$-separated set with respect to $\tilde{d}$, where $\pi$ and $\tilde{d}$ are as in Section \ref{sec: quotient}. Moreover, since $\pi_{|X_\xi}$ is a bijection and recalling that $\tilde{f}=f\circ (\pi_{|X_\xi})^{-1}$ and $\tilde{\psi}_t( \pi (x) )= \pi(\psi_t (x))$ for every $x\in X_\xi$,
\begin{displaymath}
 \sum_{x \in E} \exp\left({\int_0^t f(\psi_s(x))ds}\right)= \sum_{\tilde{x} \in \pi(E)} \exp\left({\int_0^t \tilde{f}(\tilde{\psi}_s(\tilde{x}))ds}\right).
\end{displaymath}
Therefore,
\begin{displaymath}
\hat{Z}_s(\psi_{|X_\xi},f,\varepsilon,\delta,T)\leq Z(\tilde{\psi},\tilde{f},\varepsilon,T).
\end{displaymath}
Thus, taking logarithm on both sides of the inequality, dividing by $T$ and making $T\to+\infty$, $\varepsilon\to 0^+$ and $\delta \to 0^+$ it follows that
\begin{displaymath}
\hat{P}_s(\psi_{|X_\xi},f)\leq P(\tilde{\psi},\tilde{f}),
\end{displaymath}
which combined with the initial observation concludes the proof of the lemma.
\end{proof}

Putting together Lemmas \ref{lem: Ps psi and Ps psi restrita} and \ref{lem: Ps psi tilde and Ps psi restrita} we conclude that
\begin{corollary} \label{cor: Phats equal P tilde}
\begin{displaymath}
 \hat{P}_s(\psi,f) =P(\tilde{\psi},\tilde{f}).
\end{displaymath}
\end{corollary}

\subsection{Conclusion of the proof} It follows from Theorem \ref{theo: pressure continuous}, Lemma \ref{lemma: ineq conjugacy}, Lemma \ref{lemma: Pover restric leq Pover psi}, Lemma \ref{lemma: lower upper}, Lemma \ref{lemma: Pover x Phat} and Corollary \ref{cor: Phats equal P tilde}, respectively, that
\begin{displaymath}
\begin{split}
P(\tilde{\psi},\tilde{f})&= \overline{P}_r(\tilde{\psi},\tilde{f}) \leq \overline{P}_r(\psi_{|X_\xi},f)\\
&\leq \overline{P}_r(\psi,f) \leq \overline{P}_s(\psi,f) \\
&\leq  \hat{P}_s(\psi,f) =P(\tilde{\psi},\tilde{f}).
\end{split}
\end{displaymath}
In particular,
\begin{displaymath}
P(\tilde{\psi},\tilde{f})=\overline{P}_r(\psi,f)=\overline{P}_s(\psi,f)=\hat{P}_s(\psi,f).
\end{displaymath}
Moreover, by Lemma \ref{lemma: Pover x Phat} and Lemma \ref{lemma: lower upper} we obtain that
\begin{displaymath}
\overline{P}_r(\psi,f)\leq \hat{P}_r(\psi,f)\leq \hat{P}_s(\psi,f).
\end{displaymath}
Thus,
\begin{displaymath}
P(\tilde{\psi},\tilde{f})=\overline{P}_r(\psi,f)=\overline{P}_s(\psi,f)=\hat{P}_s(\psi,f)=\hat{P}_r(\psi,f).
\end{displaymath}

Now, from \cite[Equations (5.5) and (6.1)]{ACS17} (see also the conclusion of the proof of \cite[Theorem C]{ACS17} on p. 854) we know that
\begin{displaymath}
P(\tilde{\psi},\tilde{f})=\sup_{\nu\in \mathcal{M}_{\tilde{\psi}}(X)}\left\{ h_\nu(\tilde{\psi}_1)+\int \tilde{f}d\nu\right\}=\sup_{\mu\in \mathcal{M}_\psi(X)}\left\{ h_\mu(\psi_1)+\int fd\nu\right\}
\end{displaymath}
where $\tilde{f}= f\circ H^{-1}$ and $\tilde{\psi}_1$ and $\psi_1$ stand for the time one maps of the semiflows $\tilde{\psi}$ and $\psi$, respectively. Consequently,
\begin{displaymath}
\overline{P}_s(\psi,f)=\overline{P}_r(\psi,f)= \hat{P}_s(\psi,f)= \hat{P}_r(\psi,f)=\sup_{\mu\in \mathcal{M}_\psi(X)}\left\{ h_\mu(\psi_1)+\int fd\nu\right\}.
\end{displaymath}
completing the proof of Theorem \ref{theo: var princ}.
\qed


\medskip{\bf Acknowledgements.} L. B. was partially supported by a CNPq-Brazil PQ fellowship under Grant No. 306484/2018-8.


\end{document}